\documentclass[11pt]{amsart}
\usepackage{amsmath}
\usepackage{amsfonts}
\usepackage{amssymb}
\newtheorem{thm}{Theorem}

\newtheorem{cor}[thm]{Corollary}

\newtheorem{lemma}[thm]{Lemma}

\newcommand{\bb}[1]{\mathbb{#1}}
\newcommand{\cl}[1]{\mathcal{#1}}

\begin{document}

\title[Agler's Factorization]{An Operator Algebra Proof of the Agler and Nevanlinna Factorization Theorems}

\author[S.~Lata]{Sneh Lata}
\address{Department of Mathematics, University of Houston,
Houston, Texas 77204-3476, U.S.A.}
\email{snehlata@math.uh.edu}

\author[M.~Mittal]{Meghna Mittal}
\address{Department of Mathematics, University of Houston,
Houston, Texas 77204-3476, U.S.A.}
\email{Mittal@math.uh.edu}

\author[V.~I.~Paulsen]{Vern I.~Paulsen}
\address{Department of Mathematics, University of Houston,
Houston, Texas 77204-3476, U.S.A.}
\email{vern@math.uh.edu}

\thanks{}

\begin{abstract}
We give a short direct proof of the Agler and Nevanlinna factorization theorems that uses the Blecher-Ruan-Sinclair characterization of operator algebras.
The key ingredient of this proof is an operator algebra factorization theorem.  
Our proof provides some additional information about these factorizations in the case of polynomials.
\end{abstract}

\maketitle


\section{Introduction} 

Given an analytic function $f$ on the open unit disk $\bb D,$ Nevanlinna proved that the supremum of $f$ over the disk is less than or equal to one if and only if the function $\frac{1 - f(z)\overline{f(w)}}{1-z \overline{w}}$ is a positive definite function on $\bb D.$ This latter condition is equivalent to the existence of a positive definite function $K: \bb D^2 \to \bb C$ which is analytic in $z$ and coanalytic in $w$ such that $1 - f(z) \overline{f(w)} = (1 - z \overline{w}) K(z,w).$  Here we are following the standard usage of calling a function {\em positive definite} if for every choice of finitely many points the matrix, $(K(z_i,z_j))$ is positive semidefinite.
Later this result was generalized to a parallel result for analytic matrix-valued functions on the disk whose supremum norm was less than or equal to one.

A remarkable extension of this result to more than one variable was given by Agler\cite{Ag}. To explain his result we need to first introduce what has come to be known as the {\em Schur-Agler space} of analytic functions on a polydisk. Given a natural number $N$ and $I= (i_1, \ldots, i_N) \in \bb N^N$ we set $z^I = z_1^{i_1} \cdots z_N^{i_N},$ so that every analytic function $f: \bb D^N \to \bb C$ can be written as a power series, $f(z) = \sum_{I} a_I z^I.$ 
If $T= (T_1, \ldots, T_N)$ is an $N$-tuple of operators on some Hilbert space $\cl H$ which pairwise commute and satisfy $\|T_i\| < 1, i=1, \ldots N,$ then we will call $T$ a {\em commuting N-tuple of strict contractions.} It is easily seen that if $T$ is a commuting $N$-tuple of strict contractions then the power series $f(T) = \sum_{I} a_I T^I$ converges and defines a bounded operator on $\cl H.$  The {\em Schur-Agler space} denoted by $H_{u}(\bb D^N)$ is defined to be the set of analytic functions on $\bb D^N$ such that
$\|f\|_{u} = \sup \{ \|f(T)\| \}$ is finite, where the supremum is taken over all sets of commuting N-tuples of strict contractions and all Hilbert spaces. In fact, the same supremum is attained by restricting to all commuting N-tuples of strict contractions on a fixed separable infinite dimensional Hilbert space.
It is fairly easy to see that $H_{u}(\bb D^N)$ is a Banach algebra in the norm $\| \cdot \|_{u}.$
The set of such functions with $\|f\|_{u} \le 1$ is called the {\em Schur-Agler class.}

Given a vector space $V$, we let $M_{m,n}(V)$ denote the set of $m \times n$ matrices with entries from $V.$  Given $F=(f_{i,j}) \in M_{m,n}(H_{u}(\bb D^N)),$ we set $\|F\|_{u} = \sup \{ \|(f_{i,j}(T))\| \},$ where the supremum is again over all commuting N-tuples of strict contractions and the norm of $(f_{i,j}(T))$ is computed by regarding the operator matrix as an operator from the direct sum of $n$ copies of the Hilbert space to the direct sum of $m$ copies of the Hilbert space.

Note that when the Hilbert space is one-dimensional, then every commuting N-tuple of strict contractions $T$ is of the form $T= z = (z_1, \ldots, z_N) \in \bb D^N,$ so that $\|f\|_{\infty} = \sup \{ |f(z)|: z \in \bb D^N \} \le \|f\|_{u}$ and hence, $H_{u}(\bb D^N) \subseteq H^{\infty}(\bb D^N),$ where this latter space denotes the set of bounded analytic functions on the polydisk $\bb D^N.$ When $N=1,2$ it is known that these two spaces of functions are equal and that $\|F\|_{u} = \|F\|_{\infty},$ for all $m$ and $n$.  This is a consequence of theorems of J.von Neumann\cite{Ne}, Sz.-Nagy\cite{Na} and Ando\cite{An}.

For $N \ge 3$, it is known that these two norms are not equal. 
However, it is still unknown, for a general $N\ge 3$ if these two
Banach spaces define the same sets of functions, since by the bounded
inverse theorem, $H_{u}(\bb D^N) = H^{\infty}(\bb D^N)$ if and only if
there is a constant $K_N$ such that $\|f\|_{u} \le K_N
\|f\|_{\infty}.$ The existence of such a constant is a problem that has been open since the early
1960's. For more details on all of these ideas one can see Chapters 5 and 18 of \cite{Pa}. A note of caution, in \cite{Ag} it is stated that $H_u(\bb
D^N) \ne H^{\infty}(\bb D^N), N \ge 3,$ but what is meant is that the
norms are not equal and the question of whether or not they are equal
as sets is, indeed, still open.

Agler's factorization theorem\cite{Ag} says that $F \in
M_{m,n}(H_{u}(\bb D^N))$ with $\|F\|_{u} \le 1$ if and only if there
exist positive definite matrix-valued functions $K_i: \bb D^N \times
\bb D^N \to M_m, i = 1, \dots, N,$ which are analytic in the first N
variables and coanalytic in the second N variables, such that $I_m -
F(z)F(w)^* = \sum_{i=1}^N(1 -z_i \overline{w_i})K_i(z,w).$ When
$m=n=N=1,$ Agler's result reduces to Nevanlinna's factorization
theorem since in that case $\|f\|_{\infty} = \|f\|_{u}$ by von
Neumann's inequality\cite{Ne}. The book by Agler and McCarthy\cite{AM}
is an excellent source for further information and background on this
result.

In this paper we show that Agler's factorization result is a direct
consequence of the factorization ideas of Blecher and the third
author\cite{BP} arising from the abstract theory of operator algebras.
Our proof has the advantage of being relatively short, assuming the
abstract characterization of operator algebras\cite{BRS}, and of
giving some possibly new information on the form of the positive
definite functions $K_i$ that appear in the Agler factorization.
Our proof uses an observation of McCullough's, summarized in Theorem 3,
which essentially says that the Fejer
kernel behaves as an approximate identity in the $u$-norm as well as
in the sup-norm.

\section{Main Results}

We let $\cl P_N \subseteq H_{u}(\bb D^N)$ denote the subspace spanned by the polynomials in $N$ variables endowed with the $\| \cdot \|_{u}$ norm.
The equivalence of (i) and (ii) in the following theorem is a restatement of results found in \cite{BP} and in the book \cite[Chapter 18]{Pa}, we will sketch a proof for clarity.  The equivalence of (i) and (iii) is a useful variant of Agler's factorization result for polynomials.

It will be convenient to say that matrices, $A_1, \ldots A_m$ are of {\em compatible sizes} if the product, $A_1 \cdots A_m$ exists, that is, provided that each $A_i$ is an $n_i \times n_{i+1}$ matrix.

\begin{thm}
Let $ P=(p_{ij}) \in M_{m,n}(\cl {P}_N).$ Then the following are equivalent:
\begin{itemize}
\item[(i)]$\|P\|_{u} < 1,$
\item [(ii)] there exists an integer l, matrices of scalars $C_j, 1 \le j \le l$ with $\|C_j\| <1 \ , \ 0\leq j \leq l$ and 
diagonal matrices $D_j, 1\leq j\leq l,$ each consisting of monomials in one of the variables, $z_{i_j}$ which are compatible sizes and are such that $P(z)= C_0D_1(z_{i_1})...D_l(z_{i_1})C_l,$
\item [(iii)] there exists a positive, invertible matrix $ R \in M_{m}$ and matrices of polynomials $P_i,  0\leq i\leq N, $ such that \\
$ I_m-P(z)P(w)^*= R + P_0(z)P_0(w)^* + \sum_{i=1}^{N} (1-z_i\overline{w_i})P_i(z)P_i(w)^*$ 
where $ z=(z_1,..., z_N),  w=(w_1,..., w_N) \in \bb D^N.$
\end{itemize}
\end{thm}

\begin{proof}
We will first prove that (i) implies (ii).
Note that for $m,n \in \bb N$ and $P=(p_{ij}) \in \cl M_{m, n}(\cl P_N),\\
\|P\|_{u} = sup\{ \|(\pi(p_{ij}))\| \},$ where the supremum is taken all unital algebra homomorphisms 
$\pi : \cl P_N \to \cl B(\cl H)$ such that $\|\pi(z_i)\|\leq1$ for each $i = 1,\dots,N$ and all Hilbert spaces $\cl H.\\$
The proof that (i) implies (ii) is given in \cite{BP} and \cite{Pa}. 
To illustrate how operator algebras apply, we give the outline of the
proof.

For each $m,n \in \bb N, $ one proves that
$\|P\|_{m,n} := inf\{ \|C_0\| \dots \|C_l\| \},$ defines a norm on $M_{m,n}(\cl P_N),$ 
where the infimum is taken over all $l$ and all ways to factor $P(z)=C_0D_1(z_{i_1})\dots D_l(z_{i_l})C_l$ 
as a product of matrices of compatible sizes with $C_j, 0 \leq j \leq l$ matrices of scalars and diagonal matrices $D_j, 1 \le j \le l$ whose diagonal entries are monomials in one of the variable $z_{i_j}$

Moreover, one can verify that $\cl M_{m,n}(\cl P_N)$ with this family $\{ \|.\|_{m, n}\}_{m,n}$ of norms satisfies the axioms for  
an abstract unital operator algebra as given in \cite{BRS} and hence
by the Blecher-Ruan-Sinclair representation theorem \cite{BRS}(see also \cite{Pa}) there exists a Hilbert space 
$\cl H$ and a unital completely isometric isomorphism $\pi : \cl P_N \longrightarrow \cl B(\cl H).\\$
Thus, for every $m,n \in \bb N $ and for every $ P =(p_{ij})\in \cl M_{m,n}(\cl P_N),$ 
we have $ \|P\|_{m,n} = \|\pi(p_{ij})\|.$ However, $\|\pi(z_i)\|= \|z_i\|_{1,1} \le 1$ and so, by the remark at the beginning of the proof,  $\|P\|_{m,n} =\|(\pi(p_{ij}))\|\leq \|P\|_{u}.\\$
This completes the proof that (i) implies (ii).

We will now prove that (ii) implies (iii).
Suppose that $P$ has a factorization as in (ii). We will use induction on $l$ to prove that (iii) holds.

First, assume that $l=1$ so that
$P(z) = C_0D_1(z_{i_1})C_1.$
Then,
\begin{eqnarray}
I_m - P(z)P(w)^* &=& I_m - ( C_0D_1(z_{i_1})C_1 ){( C_0D_1(z_{i_1})C_1 )}^* \nonumber \\
                   &=& \left(I_m - C_0C_0^*\right) + C_0\left(I - D_1(z_{i_1}){D_1(w_{i_1})}^*\right)C_0^* +\nonumber \\
                   && C_0D_1(z_{i_1})\left(1 -  C_1C_1^*\right)D_1(w_{i_1})^*C_0^*.\label{a} 
\end{eqnarray}
Since $D_1(z_{i_1})$ is a diagonal matrix of monomials in $z_{i_1},$ the $(k,k)$-th diagonal entry of $I - D_1(z)D_1(w)^*$ is $1 - z_{i_1}^{n_k^1}\overline{w_{i_1}}^{n_k^1}$ for some $n_k.$ Thus,
\begin{eqnarray*}
1 - z_{i_1}^{n_k^1}\overline{w_{i_1}}^{n_k^1} &=& (1 - z_{i_1}\overline{w_{i_1}}^{n_k^1})(1 + z_{i_1}\overline{w_{i_1}} + \dots + z_{i_1}^{n_k^1-1}\overline{w_{i_1}}^{n_k^1-1})
                    \\  &=& (1 - z_{i_1}\overline{w_{i_1}})A_k^1(z_{i_1})A_k^1(w_{i_1})^*, \end{eqnarray*}
where $A_k^1(z_{i_1}) = (1, z_{i_1}, z_{i_1}^2, \ldots,
z_{i_1}^{n_k^1-1}),$ is a $1\times n_k^1$ matrix of monomials in $z_{i_1}$
Hence,
$$C_0 (1 - D_1(z_{i_1})D_1(w_{i_1})^* )C_0^* 
= (1-z_{i_1} \overline{w_{i_1}}) C_0A(z_i)A(w_i)^*C_0^*,$$
where $A(z_i)$ is the direct sum of the matrices, $A^1_k.$ 
Therefore,
\begin{equation}\label{b}
C_0(1-D_1(z)D_1(w)^*)C_0^* = (1-z_{i_1}\overline{w_{i_1}})P_{i_1}(z)P_{i_1}(w)^* , \\
\end{equation} where
$$P_{i_1}(z) = C_0A_1(z_{i_1}).$$  

Setting $P_0(z) = C_0D_1(z_{i_1})(I- C_1C_1^*)^{1/2},$ we have that 
$$I_m - P(z)P(w)^* = (I_m - C_0C_0^*) + P_0(z)P_0(w)^* + (1 - z_{i_1} \overline{w_{i_1}})P_{i_1}(z)P_{i_1}(w)^*,$$ so the form (iii) holds, when $l=1.$

We now assume that the form (iii) holds for any $P(z)$ that has a factorization of length at most $l-1,$ and assume that\\ $P(z) = C_0D_1(z_{i_1}) \cdots D_{l-1}(z_{i_{l-1}}) C_{l-1} D_{l}(z_{i_{l}}) C_{l} = C_0 D_1(z_{i_1})P_{l-1}(z),$ where $P_{l-1}(z)$ has a factorization of length $l-1.$

Note that a sum of expressions such as on the right hand side of (iii)
is again such an expression. For example, 
\begin{multline*}
R_0 + P_0(z)P_0(w)^* + \sum_{i=1}^N (1 - z_i \overline{w_i})
P_i(z)P_i(w)^* \\+ B_0 + Q_0(z)Q_0(w)^* + \sum_{i=1}^N (1-z_i
\overline{w_i}) Q_i(z)Q_i(w)^* \\= H_0 + S_0(z)S_0(w)^* + \sum_{i=1}^N
(1- z_i \overline{w_i}) S_i(z)S_i(w)^*,
\end{multline*} where $H_0 = R_0 + B_0$ is positive and invertible, and $S_i(z) = (P_i(z), Q_i(z)).$ Thus, it will be sufficient to show that $(I_m - P(z)P(w)^*)$ is a sum of expressions as above.

To this end we have that,
\begin{multline*}
I_m - P(z)P(w)^* = (I_m - C_0C_0^*) + C_0(I -
D_1(z_{i_1})D_1(w_{l_1})^*)C_0^* \\+ (C_0D_1(z_{i_1}))(I -
P_{l-1}(z)P_{l-1}(w)^*)(D_1(w_{i_1})^*C_0^*).
\end{multline*}
The first two of these three terms can be seen by the above arguments
to be of the form as on the right hand side of (iii). The quantity
$(I- P_{l-1}(z)P_{l-1}(w)^*)$ is of this form by the inductive
hypothesis. Finally, note that if $H(z,w)$ can be written as such a
sum, then so too can $Q(z)H(z,w)Q(w)^*,$ so it then follows that the
third term is also of the required form.

Finally, we will prove (iii) implies (i). 
Suppose the (iii) holds. Then for any commuting $N$-tuple of strict contractions $T = (T_1, \ldots, T_N),$ using the argument of 
uniqueness of power series, we have
\begin{align*}
 I_m - P(T)P(T)^* &= R + P_0(T)P_0(T)^* + \sum_{i=1}^NP_i(T)(1-T_iT_i^*)P_i(T)^*.
\end{align*}
Clearly, each term on the right hand side of the above inequality is positive and since $R$ is strictly positive, say $R \ge \delta I_m$ for some scalar $\delta > 0,$ we have that $I_m - P(T)P(T)^* \ge \delta I_m\\$
Therefore $ (1- \delta)I_m \ge P(T)P(T)^* ,$ which implies $\|P(T)\| \le \sqrt{1- \delta}.$
Thus, since $T$ was arbitrary,  $\|P\|_{\cl U} \le \sqrt{1- \delta} < 1,$ which proves (i).
\end{proof}

\begin{cor} Let $ P=(p_{ij}) \in M_{m,n}(\cl {P}_N).$ If $\|P\|_u < 1,$ then there exist positive definite matrix-valued functions, $K_i: \bb D^N \times \bb D^N \to M_m, 0 \le i \le N,$ whose components are polynomials such that
$$I_m - P(z)P(w)^* = K_0(z,w) + \sum_{i=1}^N (1 - z_i \overline{w_i})K_i(z,w),$$
for all $z,w \in \bb D^N.$
\end{cor}
\begin{proof} Usiing the form in (iii), we set $K_0(z,w) = R + P_0(z)P_0(w)^*$ and $K_i(z,w) = P_i(z)P_i(w)^*, 1 \le i \le N.$
\end{proof}

One of the advantages of the above factorization over Agler's form is that, since each positive definite function $K_i$ consists of polynomials, the associated reproducing kernel Hilbert spaces will be finite dimensional spaces of $\bb C^m$-valued polynomials.
  
To see the connection between the factorization occurring in our (iii) and Agler's factorization, note that each of the terms, $K_i(z,w) = P_i(z)P_i(w)^*, i \ge 2$ is a positive definite matrix-valued function that is analytic in the $z$ variables and coanalytic in the $w$ variables. If we set
$$K_1(z,w) =
 \frac{R+P_0(z)P_0(w)^*+P_1(z)P_1(w)^*}{1-z_1\overline w_1},$$
Then $K_1(z,w)$ is also a positive definite matrix-valued function that is analytic in $z$ and coanalytic in $w$ and we have that
$I_m - P(z)P(w)^* = \sum_{i=1}^N(1- z_i \overline w_i) K_i(z,w).$
Unfortunately, when written in this form, the reproducing kernel Hilbert space associated with this new function $K_1,$ will generally be an infinite dimensional Hilbert space of analytic $\bb C^m$-valued functions, even though $P(z)$ was only polynomial valued.

To pass from the above theorem for polynomials to the full version of Agler's theorem, we need to first show that functions in $M_{m,n}(H_{u}(\bb D^N))$ have nice approximations by matrices of polynomials.
Note that if we write a matrix of analytic functions 
$F \in M_{m,n}(H_{u}(\bb D^N)),  F=(f_{i,j})$ as a power series,
$F(z) = \sum_I A_I z^I,$ where $A_I \in M_{m,n}$ are scalar matrices, then for any commuting $N$-tuple of strict contractions we will have that $F(T) = \sum A_I \otimes T^I$ where by the tensor product of an $m \times n$ scalar matrix $B=(b_{i,j})$ and an operator $R \in B(\cl H)$ we mean the operator $B \otimes R = (b_{i,j}R)$ from the direct sum of $n$ copies of $\cl H$ to the direct sum of $m$ copies of $\cl H.$

The following fact is also used in Agler's proof \cite{Ag}.

\begin{lemma} Given any factorization of the form in Agler's theorem \\
$ I_m- F(z)F(w)^* = \sum_{i=1}^N(1-z_i\overline w_i)K_i(z,w),$ each of the functions $K_i(z,w)$ satisfy, $\|K_i(z,w)\|^2 \leq \frac{1}{(1-|z_i|^2)(1-|w_i|^2)} $ and hence are locally bounded in $\bb D^N.$
\end{lemma}
\begin{proof}
For $z= (z_1, \ldots , z_N) \in \bb D^N,$ we have that 
$(1- |z_i|^2)K_i(z,z) \\ \le I_m  -F(z)F(z)^* \le I_m,$
and hence, $\|K(z,z)\| \le \frac{1}{ 1 - |z_i|^2}.$

As eack $K_i$ is positive definite, using the positivity of the $2 \times 2$ block matrix corresponding to the pair of points, $z,w \in \bb D^N,$
we obtain $\|K_i(z,w)\|^2 \leq \|K_i(z,z)\| \cdot \|K_i(w,w)\| \le \frac{1}{(1-|z_i|^2)(1-|w_i|^2)}.\\$ 

This inequality shows that $\|K_i(z,w)\|$ is uniformly bounded on $(r \bb D^N) \times (r \bb D^N)$ for any $0 < r < 1,$ and hence locally bounded.
\end{proof}

\begin{thm} Let $ F\in \cl M_{m,n}(H_{u}(\bb D ^ N))$ with $ F(z) = \sum_I A_I z^I $ 
for $ z\in \bb D^N.$ Then the sequence $\{ P_n\}$ of matrices of polynomials 
$ P_n(z) = \sum_{|I|\leq n}(1 - \frac{|I|}{n+1})A_I z^I$ converges locally uniformly to F 
and $\|P_n\|_{\cl U} \le \|F\|_{\cl U}$ for each n. Conversely, if there is a sequence of $ P_n,$ matrices of polynomials, 
converging to F pointwise on $\bb D^N$ with $ \| P_n\|_{\cl U} \leq 1 $ for each n, then $ \| F\|_{\cl U}\leq 1. $
\end{thm}
\begin{proof}
Fix an $ n\in \bb N $ and consider the Fejer kernel,
\begin{center}
 $F_n(\theta) = \frac{1}{n+1}\sum_{k,l=0}^n e^{i(k-l)\theta} $    for $\theta \in [0, 2\pi]. $
\end{center}
Note that for each fixed $z \in \bb D^N$ the function $ \theta \to F(ze^{i \theta}) = F(z_1 e^{i \theta}, \ldots, z_N e^{i \theta})$ is continuous.
We define $ P_n(z) = \frac{1}{2\pi} \int_0^{2\pi}F(ze^{i\theta})F_n(e^{i\theta})d\theta \quad$ for $z\in \bb D^N,$ where the integration is in the Riemann sense.
A direct calculation shows that  $ P_n(z) = \sum_{|I|\leq n}(1 - \frac{|I|}{n+1}) A_I z_I,$ where $|I|= i_1+ \cdots + i_N.$

Next check that for a fixed commuting $N$-tuple of strict contractions, $T=(T_1, \ldots, T_N),$ on a Hilbert space $\cl H,$ the map $\theta \to F(T_1 e^{i \theta}, \ldots, T_Ne^{i \theta})$ is continuous from the interval into $B(\cl H)$ equipped with the norm topology. This follows from the fact that since we are dealing with strict contractions, $F(T e^{i \theta})$ is a norm limit of partial sums of its power series.
It now follows that 
\begin{center}
$P_n(T) = \frac{1}{2\pi}\int_0^{2\pi}F(Te^{i\theta})F_n(e^{i\theta})d\theta,$
\end{center}
where the integration is again in the Riemann sense.

Thus,   $$\|P_n(T)\| \leq \frac{1}{2\pi}\int_0^{2\pi}\|F(Te^{i\theta})\|F_n(e^{i\theta})d\theta\\
\leq \|F\|_{u}$$
and we have shown that
$\|P_n\|_{u}\leq \|F\|_{u}.$

The fact that $ P_n$ converges to F locally uniformly is a standard result for scalar valued functions. To see directly in our case note that for
$z\in \bb{D}^N,$
we have that\\
$P_n(z) = \sum_{|I|\leq n}\frac{(n+1-|I|)}{n+1}A_Iz^I
= \frac{S_0(z)+\dots+S_n(z)}{n+1},$\\
where $S_k(z) = \sum_{|I|\leq k}A_Iz^I, k=1,\dots,n.\\$
So, $P_n \longrightarrow F$ locally uniformly on $\bb D^N.$

For the converse, let $\{P_n\}$ be a sequence of $\cl M_{m,n}$ valued polynomials with
 $\|P_n\|_{u} \le 1$ and converging to F pointwise on $\bb{D}^N.\\$
For each n, $\|P_n\|_\infty \leq \|P_n\|_{u} \le 1.$
This implies that there exist a subsequence $\{P_{n_k}\}$ which converges to a function 
$G\in \cl M_{m,n}(H^\infty(\bb{D}^N))$ in the weak* topology and, hence, that 
$\{P_{n_k}\}$ converges to G uniformly on compact subsets of $\bb{D}^N.\\$
Thus, G = F and  $\{P_{n_k}\}$ converges to F uniformly on compact subsets of $\bb D^N.\\$
If we now take $T=(T_1,\dots, T_N)$ a commuting $N$-tuple of strict contractions, then there is an $r < 1,$ so that $\|T_i\| \le r$ for all $i.$
Since the polynomials converge to $F$ uniformly on the closed polydisk of radius $\frac{1+r}{2}$ it follows by the Riesz functional calculus that 
$P_{n_k}(T_1,\dots,T_N)\longrightarrow F(T_1,\dots,T_N)$ in norm.\\ 

Therefore, $\|F(T)\| = \lim_{k\to \infty}\|P_{n_k}(T)\| \leq 1$ and, hence, $\| F\|_{u} \leq 1.$
\end{proof}

We can now prove Agler's theorem.

\begin{thm}[Agler's Factorization Theorem]
Let $F \in \cl M_{m,n}(\cl H^\infty(\bb D^N)).$ Then $\|F\|_{u} \leq 1$ if and only if there exist
positive definite functions $K_i:\bb D^N\times \bb D^N \to M_m,1 \leq i\leq N,$ analytic in the first variable and coanalytic in the second variable, such that\\
\begin{center}
$1 - F(z)F(w)^* = \sum_{i=1}^N(1 -z_i\overline{w_i})K_i(z,w)$
\end{center}
where $ z = (z_1,\dots,z_N), w = (w_1,\dots,w_N)\in \bb D^N.$ 
\end{thm}

\begin{proof}
Let $\|F\|_{u} \leq 1.$  By Theorem~3, there exists a sequence $\{P_n\}$ of matrices
of polynomials such that $P_n$ converges to F locally uniformly on $\bb D^N$ with $\|P_n\|_{u} < 1$ for each n. 

Since $\|P_n\|_{u} < 1,$ by Theorem~1 and the remarks following, there exist
positive definite functions $K_i^n:\bb D^N \times \bb D^N\to M_m, 1\leq i\leq N,$ analytic in the first variable and coanalytic in the second, such that
\begin{center}
$1 - P_n(z)P_n(w)^* = \sum_{i=1}^N(1 - z_i\overline{w_i})K_i^n(z,w), $
\end{center}
where $ z = (z_1,\dots,z_N), w = (w_1,\dots,w_N).$

Moreover by Proposition 3, for each $ (z,w) \in \bb D^N \times \bb D^N$ and for each n and i,  $\|K_i^n(z,w)\|^2 \leq \frac{1}{(1-|z_i|^2)(1-|w_i|^2)}.$
So, for each fixed $ i\in\{1,\dots,N\}$ we have a locally uniformly bounded sequence $\{K_i^n : n\in \bb N\}$ of 
analytic-coanalytic functions.  Thus, each of the $m^2$ functions that make up the coefficients of the matrix-valued functions, 
$K_i^n : \bb D^N \times \bb D^N \longrightarrow M_m,$ is locally uniformly bounded.
 
Hence, applying Montel's theorem($m^2N$ times) there exists a subsequence $\{ n_k \}$ and analytic-coanalytic functions $K_i : \bb D^N \times \bb D^N \to M_m, 1 \le i \le N$ such that
the subsequence $\{K_i^{n_k}\}$ converges to $K_i$ locally uniformly on $\bb D^N \times \bb D^N.\\$ Since each function in the subsequence is positive definite, it follows that each $K_i$ is positive definite.
Taking limits, we have that
$$I_m - F(z) F(w)^* = \lim_k (I_m - P_{n_k}(z)P_{n_k}(w)^*) = \sum_{i=1}^N(1 - z_i \overline w_i) K_i(z,w).$$
The proof of the converse is identical to the argument given by Agler\cite{Ag}.
We briefly recall it for completeness. Assume that we are given that
 $$I_m - F(z)F(w)^* = \sum_{i=1}^N(1-z_i \overline{w_i})K_i(z,w),$$ where each $K_i$ is a positive definite $M_m$-valued function that is analytic in $z$ and coanalytic in $w$.

From the general theory of reproducing kernel Hilbert spaces, it follows that there exists separable Hilbert spaces $\cl H_i, 1 \le i \le N$ and analytic functions, $F_i: \bb D^N \to B(\cl H_i, \bb C^m), 1 \le i \le N,$ such that $K_i(z,w) = F_i(z)F_i(w)^*, 1 \le i \le N.$  For those familiar with this theory, let $\cl H_i$ be the reproducing kernel Hilbert space of $\bb C^m$-valued functions constructed from $K_i$ and let $F_i(z): \cl H_i \to \bb C^m$ be evaluation at $z.$ Choosing an orthonormal basis for $\cl H_i,$ one may regard $F_i(z)$ as an $m \times \infty$ matrix of analytic functions on $\bb D^N.$
One then verifies that for any commuting $N$-tuple of strict contractions,
$$I_m - F(T)F(T)^* = \sum_{i=1}^NF_i(T)(I - T_iT_i^*)F_i(T)^*  \ge 0,$$ and so, $\| F(T) \| \le 1.$  Since $T$ was arbitrary, we have that $\|F\|_{u} \le 1.$
\end{proof}

As we showed in the above proof, each analytic-coanalytic positive definite function, $K_i: \bb D^N \times \bb D^N \to M_m,$ can be factored as $K_i(z,w) = F_i(z)F_i(w)^*,$ where $F_i$ is an $m \times \infty$ matrix of analytic functions on $\bb D^N.$ Often Agler's factorization theorem is written in this equivalent form, i.e., as
$$I_m - F(z)F(w)^* = \sum_{i=1}^N (1 -z_i \overline{w_i}) F_i(z)F_i(w)^*.$$

\end{document}